\documentclass[letterpaper,12pt]{article}

\usepackage{fullpage}

\usepackage{graphicx}
\usepackage[numbers,sort]{natbib}

\newtheorem{theorem}{Theorem}[section]
\newtheorem{lemma}{Lemma}[section]

\newtheorem{proposition}{Proposition}[section]

\newenvironment{proof}{\par\noindent{\bf Proof.}}{\hfill%
\qed\bigskip}

\def\be{\begin{equation}}
\def\ee{\end{equation}}
\def\ba{\begin{array}}
\def\ea{\end{array}}
\newcommand{\nc}{\newcommand}
\def\sqr#1#2{{\vcenter{\vbox{\hrule height.#2pt
        \hbox{\vrule width.#2pt height#1pt \kern#2pt
        \vrule width.#2pt}
        \hrule height.#2pt}}}}
\nc{\qed}{\hfill $\mathchoice\sqr77\sqr77\sqr{2.1}3\sqr{2.1}3 $}

\usepackage{amsfonts}
\usepackage{amsmath}
\nc{\bea}{\begin{eqnarray}}
\nc{\eea}{\end{eqnarray}}
\nc{\bean}{\begin{eqnarray*}}
\nc{\eean}{\end{eqnarray*}}
\nc{\nn}{\nonumber}
\nc{\cA}{{\cal A}}
\nc{\cB}{{\cal B}}
\nc{\cC}{{\cal C}}
\nc{\cD}{{\cal D}}
\nc{\cE}{{\cal E}}
\nc{\cG}{{\cal G}}
\nc{\cF}{{\cal F}}
\nc{\cH}{{\cal H}}
\nc{\cI}{{\cal I}}
\nc{\cJ}{{\cal J}}
\nc{\cK}{{\cal K}}
\nc{\cL}{{\cal L}}
\nc{\cM}{{\cal M}}
\nc{\cN}{{\cal N}}
\nc{\cO}{{\cal O}}
\nc{\cP}{{\cal P}}
\nc{\cQ}{{\cal Q}}
\nc{\cR}{{\cal R}}
\nc{\Yildirim}{Y{\i}ld{\i}r{\i}m}
\nc{\cS}{{\cal S}}
\nc{\cT}{{\cal T}}
\nc{\cU}{{\cal U}}
\nc{\cV}{{\cal V}}
\nc{\tx}{{\tilde x}}
\nc{\la}{{\langle}}
\nc{\ra}{{\rangle}}
\def\R{\mathbb{R}}

\title{Global Solutions of Nonconvex Standard Quadratic Programs via Mixed Integer Linear Programming Reformulations}

\author{Jacek Gondzio\thanks{School of Mathematics, The University of Edinburgh, 
James Clerk Maxwell Building, EH9 3FD Edinburgh, United Kingdom ({\tt J.Gondzio@ed.ac.uk})} \and E. Alper \Yildirim\thanks{Department of
Industrial Engineering, Ko{\c c} University, 34450 Sar{\i}yer, Istanbul, Turkey ({\tt alperyildirim@ku.edu.tr}).}}

\date{\today}

\begin{document}

\maketitle

\vspace{-.5in}

\begin{abstract}
A standard quadratic program is an optimization problem that consists of minimizing a (nonconvex) quadratic form over the unit simplex. We focus on reformulating a standard quadratic program as a mixed integer linear programming problem. We propose two alternative mixed integer linear programming formulations. Our first formulation is based on casting a standard quadratic program as a linear program with complementarity constraints. We then employ binary variables to linearize the complementarity constraints. For the second formulation, we first derive an overestimating function of the objective function and establish its tightness at any global minimizer. We then linearize the overestimating function using binary variables and obtain our second formulation. For both formulations, we propose a set of valid inequalities. Our extensive computational results illustrate that the proposed mixed integer linear programming reformulations significantly outperform other global solution approaches. On larger instances, we usually observe improvements of orders of magnitude.
\end{abstract}


\bigskip



\section{Introduction}
 
A standard quadratic program is an optimization problem in which a (nonconvex) homogeneous quadratic function is minimized over the unit simplex. An instance of a standard quadratic program is given by  
\[
 \textrm{(StQP)} \quad \nu(Q) := \min\limits_{x \in \Delta_n} x^T Q x,
\]
 where $Q \in \cS^n$ and $\cS^n$ denotes the set of $n \times n$ real symmetric
 matrices, $x \in \R^n$, and $\Delta_n$
 denotes the unit simplex in $\R^n$ given by
 \be \label{def_unit_simp}
 \Delta_n := \left\{x \in \R^n_+: e^T x= 1\right\},
 \ee
 where $e \in \R^n$ is the vector of all ones and $\R^n_+$ denotes the
 nonnegative orthant in $\R^n$.
 
We remark that having a quadratic form in the objective function is not restrictive since the problem of minimizing a nonhomogeneous quadratic function over the unit simplex can be reformulated in the form of (StQP) using the following identity:
 \[
  x^T Q x + 2 c^T x = x^T (Q + e c^T + c e^T) x, \quad \textrm{for each}~ x \in \Delta_n.
 \] 
 
Standard quadratic programs arise in a variety of applications ranging from the classical portfolio optimization problem~\cite{HM1952} to population genetics~\cite{K61}; from quadratic resource allocation~\cite{IN88} to selection replicator dynamics and evolutionary game theory~\cite{Bom02a}. For a given matrix $Q \in \cS^n$, $Q$ is copositive if and only if $\nu(Q) \geq 0$. Therefore, standard quadratic programs can be used to check if a matrix is copositive. We refer the reader to the paper~\cite{Bom98} for other applications, in which the term standard quadratic program was coined. The maximum stable set problem in graph theory~\cite{MS65} and its weighted version~\cite{Gib97} can be formulated as instances of (StQP), which implies that (StQP) is, in general, NP-hard. 

There is an extensive amount of literature on standard quadratic programs. In this paper, we are concerned with computing a global solution of (StQP). We therefore restrict our literature review to the exact solution approaches in the literature. All of these approaches, in general, are based on a branch-and-bound scheme and differ only in terms of the subroutines used for computing upper and lower bounds, and subdividing the feasible region. For instance, a DC (difference of two convex functions) programming approach is employed in~\cite{Bom02b} to compute a lower bound and a local optimization method is used to find an upper bound. Using a relation between global solutions of (StQP) and the set of cliques in an associated graph, referred to as the {\em convexity graph} (see Section~\ref{prop_opt_sol}), a branch-and-bound method based on an implicit enumeration of cliques is proposed in~\cite{ST08}. More recently, another branch-and-bound method was proposed in~\cite{LMP17}, in which both convex envelope estimators and polyhedral underestimators are employed for computing lower bounds and an implicit enumeration of the KKT points using the relation with the set of cliques in the convexity graph is utilized. In another recent paper~\cite{BLST16}, a set of cutting planes is proposed in the context of a spatial branch-and-bound scheme.

Standard quadratic programs can also be solved by finite branch-and-bound methods proposed for solving more general nonconvex quadratic programming problems (see, e.g.,~\cite{BV08,CB12}). These approaches are based on an implicit enumeration of the complementarity constraints in the KKT conditions. The resulting subproblems are approximated by semidefinite relaxations or by polyhedral semidefinite relaxations. By a simple manipulation of the KKT conditions, a general quadratic program can be formulated as a linear program with complementarity constraints (LPCC) (see, e.g.,~\cite{Hu2012} and the references therein) and the resulting LPCC can be solved by an enumerative scheme such as branch-and-bound. An LPCC can also be formulated as a mixed integer linear programming (MILP) problem and can be solved using Benders decomposition~\cite{Hu2012} or by branch-and-cut~\cite{YMP16}. A similar MILP formulation is proposed in~\cite{XVZ15} under the assumption of a bounded feasible region. Alternatively, using the completely positive reformulation of (StQP) (see, e.g.,~\cite{Bom_etal_00}), adaptive inner and outer polyhedral approximations of completely positive programs can be employed~\cite{BD09}. Clearly, one can also use general purpose nonlinear programming solvers such as {BARON}~\cite{TS05} and {COUENNE}~\cite{BLL09}.

In this paper, we propose globally solving a standard quadratic program by reformulating it as a mixed integer linear programming (MILP) problem. We choose MILP reformulations due to the existence of powerful state-of-the-art MILP solvers such as {CPLEX}~\cite{cplex} and {GUROBI}~\cite{gurobi}. We propose two different MILP reformulations. Our first formulation is based on reformulating (StQP) as a linear program with complementarity constraints and linearizing the complementarity constraints by using binary variables and big-$M$ constraints. We discuss how to obtain valid bounds for the big-$M$ parameters by exploiting the structure of (StQP). The second formulation is obtained by replacing the quadratic objective function in (StQP) by an overestimating function given by the maximum of a finite number of linear functions associated with the positive components of a feasible solution, referred to as the support. We show that the overestimating function is exact at all KKT points of (StQP), which leads to the second MILP formulation by introducing binary variables for modeling the support of a feasible solution. We further show that our second MILP formulation is, in fact, an exact relaxation of the first one. Furthermore, using the relation between the support of a global minimizer of (StQP) and the set of cliques in the convexity graph, we propose a set of valid inequalities for both of our MILP formulations. We conduct extensive computational experiments to assess the performances of our MILP formulations in comparison with several other global solution approaches. The computational results indicate that the proposed MILP formulations consistently outperform other global solution approaches. Furthermore, on especially large instances, we observe improvements of orders of magnitude.

Our work is related to the previous work on reformulations of a quadratic program as an instance of a linear program with complementarity constraints (LPCC)~\cite{Hu2012,XVZ15}. For a general quadratic program, the paper~\cite{Hu2012} proposes a two-stage LPCC approach. In the first stage, an LPCC is solved to determine if the quadratic program is bounded below, in which case a second LPCC is formulated to compute a global solution. The resulting complementarity problems are formulated as MILP problems and solved using a parameter-free approach via Benders decomposition, which eliminates the need for big-$M$ parameters~\cite{MPBK08} (see~\cite{YMP16} for a branch-and-cut approach). In contrast, the paper~\cite{XVZ15} explicitly uses big-$M$ parameters. By using a Hoffman type error bound, the authors show that there exists a valid upper bound for the big-$M$ parameters under the assumption of a bounded feasible region. They give a closed form expression of this bound for (StQP). Our first MILP formulation, which is based on a similar approach as in these previous MILP formulations, also employs big-$M$ parameters as in~\cite{XVZ15}. In contrast with their approach, we exploit the specific structure of (StQP) in an attempt to obtain much tighter bounds for big-$M$ parameters. Furthermore, our second MILP formulation is based on specifically taking advantage of the particular structure of (StQP). Therefore, in contrast with the previous approaches in the literature, we propose stronger MILP formulations for a more specific class of quadratic programs.
 
This paper is organized as follows. In Section~\ref{notation}, we briefly review our notation. Section~\ref{prelim} discusses several useful properties of standard quadratic programs. We present our MILP formulations as well as a set of valid inequalities in Section~\ref{milp_form}. Section~\ref{comp} is devoted to the results of our computational experiments. We conclude the paper in Section~\ref{conc}.

\subsection{Notation} \label{notation}

We use $\R^n, \R^n_+$, and $\cS^n$ to denote the $n$-dimensional Euclidean space, the nonnegative orthant, and the space of $n \times n$ real symmetric matrices, respectively. For $u \in \R^n$, we denote its $j$th component by $u_j,~j = 1,\ldots,n$. Similarly, $U_{ij}$ denotes the $(i,j)$ entry of a matrix $U \in \cS^n,~i = 1,\ldots,n;~j = 1,\ldots,n$. We denote the unit simplex in $\R^n$ by $\Delta_n$. For any $U \in \cS^n$ and $V \in \cS^n$, the trace inner product of $U$ and $V$ is given by $\langle U, V \rangle := \sum\limits_{i=1}^n \sum\limits_{j = 1}^n U_{ij} V_{ij}$. The unit vectors in $\R^n$ are denoted by $e_j,~j = 1,\ldots,n$. We reserve $e \in \R^n$ and $E = e e^T \in \cS^n$ for the vector of all ones and the matrix of all ones, respectively. We use $0$ to denote the real number zero, the vector of all zeroes as well as the matrix of all zeroes in the appropriate dimension, which will always be clear from the context. We use $\textrm{conv}(\cdot)$ to denote the convex hull. We define the following convex cones in $\cS^n$:
\begin{eqnarray}
\label{def_Nn}
\cN^n & = & \left\{M \in \cS^n: M_{ij} \geq 0, \quad i = 1,\ldots,n;~j = 1,\ldots,n\right\}, \\
\label{def_Sn}
\cS^n_+ & = & \left\{M \in \cS^n: u^T M u \geq 0, \quad \forall~u \in \R^n \right\}, \\
\label{def_COPn}
{\cal COP}^n & = & \left\{M \in \cS^n: u^T M u \geq 0, \quad \forall~u \in \R^n_+ \right\}, \\
\label{def_CPn}
{\cal CP}^n & = & \textrm{conv}\left\{u u^T: u \in \R^n_+\right\}, \\
\label{def_Dn}
{\cal DNN}^n & = & \cS^n_+ \cap \cN^n,
\end{eqnarray}
namely, the cone of component-wise nonnegative matrices, the cone of positive semidefinite matrices, the cone of copositive matrices, the cone of completely positive matrices, and the cone of doubly nonnegative matrices, respectively. The following relations easily follow from these definitions:
\begin{equation} \label{inclusions}
{\cal CP}^n \subseteq {\cal DNN}^n \subseteq {\cal COP}^n.
\end{equation}

\section{Preliminaries} \label{prelim}

In this section, we review several basic properties of standard quadratic programs that will be useful in the subsequent sections. We remark that these results can be found in the literature (see, e.g.,~\cite{Bom98,BLT08}). We include proofs of some of these results for the sake of completeness.

\subsection{Optimality Conditions}

Since a standard quadratic program has linear constraints, constraint qualification is satisfied at every feasible solution. Given an instance of (StQP), if $x \in \Delta_n$ is an optimal solution, then there exist $s \in \R^n$ and $\lambda \in \R$ such that the following KKT conditions are satisfied:
\begin{eqnarray} \label{kkt}
Q x - \lambda e - s & = & 0, \label{eq1}\\
e^T x & = & 1, \label{eq2} \\
x & \in & \R^n_+, \label{eq3} \\
s & \in & \R^n_+, \label{eq4} \\
x_j s_j & = & 0, \quad j = 1,\ldots,n. \label{eq5}
\end{eqnarray}

For an instance of (StQP), $x \in \Delta_n$ is said to be a {\em KKT point} if there exist $s \in \R^n$ and $\lambda \in \R$ such that the conditions (\ref{eq1}) -- (\ref{eq5}) are satisfied. 

\subsection{Properties of $\nu(Q)$} 

The following lemma presents several useful properties about the optimal value function $\nu(\cdot)$.

\begin{lemma} \label{lem1}
For any $Q \in \cS^n$, $Q_1 \in \cS^n$, $Q_2 \in \cS^n$, and $\gamma \in \R$, the following relations are satisfied:
\begin{enumerate}
\item[(i)] $\nu(Q + \gamma ee^T) = \nu(Q) + \gamma$.
\item[(ii)] If $Q_1 - Q_2 \in \cN^n$, then $\nu(Q_1) \geq \nu(Q_2)$.
\item[(iii)] If $Q$ is a diagonal matrix with strictly positive diagonal entries $Q_{11},\ldots,Q_{nn}$, then 
\[
\nu(Q) = \frac{1}{\sum\limits_{k=1}^n (1/Q_{kk})}.
\]
\item[(iv)] Let $\gamma_0 = \min\limits_{1\leq i \leq j \leq n} Q_{ij}$ and $\gamma_1 = \min\limits_{k=1,\ldots,n} Q_{kk}$. Then, $\gamma_0 \leq \nu(Q) \leq \gamma_1$. 
\end{enumerate}
\end{lemma}
\begin{proof}
\begin{enumerate}
\item[(i)] For any $Q \in \cS^n$ and any $\gamma \in \R$, we have
\[
 \nu(Q + \gamma e e^T) = \min\{x^T (Q + \gamma e e^T) x: x \in \Delta_n\} = \gamma + \min\{x^T Q x: x \in \Delta_n\} = \nu(Q) + \gamma.
\]
\item[(ii)] Let $Q_1 - Q_2 \in \cN^n$. Then, for any $x \in \Delta_n$, we have $x^T (Q_1 - Q_2) x \geq 0$, which implies that 
\[
x^T Q_1 x \geq x^T Q_2 x, \quad \forall~x \in \Delta_n,
\]
from which the assertion follows.
\item[(iii)] If $Q$ is a diagonal matrix with strictly positive diagonal entries $Q_{11},\ldots,Q_{nn}$, then a simple manipulation of the KKT conditions (\ref{eq1}) -- (\ref{eq5}) reveals that the unique KKT point satisfies
\[
x_k = \frac{1/Q_{kk}}{\sum\limits_{k=1}^n (1/Q_{kk})}, \quad k = 1,\ldots,n.
\]
Substituting this solution in the objective function yields the result.
\item[(iv)] Let $\gamma_0 = \min\limits_{1\leq i \leq j \leq n} Q_{ij}$. Then, $Q - \gamma_0 ee^T \in \cN^n$, which implies that $0 \leq \nu(Q - \gamma_0 ee^T) = \nu(Q) - \gamma_0$, where we used (i). Therefore, $\gamma_0 \leq \nu(Q)$. Furthermore, since $e_k \in \Delta_n$ for each $k = 1,\ldots,n$, we have $\nu(Q) \leq \min\limits_{k=1,\ldots,n} e_k^T Q e_k = \min\limits_{k=1,\ldots,n} Q_{kk} = \gamma_1$.
\end{enumerate}
\end{proof}

\subsection{Properties of An Optimal Solution} \label{prop_opt_sol}

In this section, we present a useful relation between an optimal solution of a standard quadratic program and a related graph. 

Given an instance of (StQP) with $Q \in \cS^n$, we can associate with it an undirected graph $G = (V,E)$, called the {\em convexity graph of $Q$}, where $V = \{1,2,\ldots,n\}$ with node $j$ corresponding to the vertex of the unit simplex $e_j,~j = 1,\ldots,n$. There is an edge between node $i$ and node $j$ if the restriction of the quadratic form $x^T Q x$ to the edge of the unit simplex between the vertices $e_i$ and $e_j$ is strictly convex, i.e., 
\[
E = \{(i,j): Q_{ii} + Q_ {jj} - 2 Q_{ij} > 0, \quad 1 \leq i < j \leq n\}.
\]

For $x \in \Delta_n$, we introduce the following index sets:
\begin{eqnarray}
\label{def_Px}
{\cal P}(x) & = & \left\{ j \in \{1,\ldots,n\}: x_j > 0 \right\}, \label{def_P} \\
\label{def_Zx}
{\cal Z}(x) & = & \left\{ j \in \{1,\ldots,n\}: x_j = 0 \right\}. \label{def_Z} 
\end{eqnarray}

The indices in ${\cal P}(x)$ are referred to as the {\em support} of $x$. 
For a given undirected graph $G = (V,E)$, a set $C \subseteq V$ of nodes is called a {\em clique} if each pair of nodes is connected by an edge. Similarly, a set $S \subseteq V$ of nodes is called a {\em stable set} if no two nodes in $S$ are connected by an edge. The following theorem~\cite{ST08} establishes a useful connection between the support of a global solution of (StQP) and the convexity graph $G = (V,E)$.

\begin{theorem}[Scozzari and Tardella, 2008] \label{clique_thm}
Given an instance of (StQP), let $G = (V,E)$ denote the convexity graph of $Q$. Then, there exists a globally optimal solution $x^* \in \Delta_n$ of (StQP) such that the nodes corresponding to the indices in ${\cal P}(x^*)$ (i.e., the support of $x^*$) in $G$ form a clique (or, equivalently, a stable set in the complement of $G$). 
\end{theorem} 

\subsection{Lower Bounds on $\nu(Q)$}

In this section, given an instance of (StQP), we review two lower bounds on $\nu(Q)$.

\subsubsection{A Simple Lower Bound}

We start with a simple lower bound on $\nu(Q)$. By Lemma~\ref{lem1}(iv), 
\[
\nu(Q) \geq \gamma_0 = \min_{1\leq i \leq j \leq n} Q_{ij},
\]
with equality if there exists $k \in \{1,\ldots,n\}$ such that $\gamma_0 = \gamma_1 = Q_{kk}$.

This lower bound can be slightly sharpened if the minimum entry of $Q$ is not on the main diagonal, i.e. if $\gamma_0 < \gamma_1$. In this case, we have $Q - \gamma_0 e e^T \in \cN^n$ with strictly positive diagonal elements, which can be decomposed as $Q - \gamma_0 e e^T = D + F$, where $D \in \cN^n$ and $F \in \cN^n$ are such that $D$ is a diagonal matrix with strictly positive entries $D_{kk} = Q_{kk} - \gamma_0,~k = 1,\ldots,n$ along the main diagonal, and all diagonal entries of $F$ are equal to zero. Since $(Q - \gamma_0 e e^T) - D = F \in \cN^n$, it follows by Lemma~\ref{lem1}(i), (ii), and (iii) that 
\[
\nu(Q - \gamma_0 e e^T)  = \nu(Q) - \gamma_0 \geq \nu(D) = \frac{1}{\sum\limits_{k=1}^n (1/D_{kk})} = \frac{1}{\sum\limits_{k=1}^n (1/\left(Q_{kk} - \gamma_0\right))}.
\]
This gives rise to the following lower bound on $\nu(Q)$ (see, e.g.,~\cite{BLT08}):
\[
  \textrm{(LB1)} \quad \nu(Q) \geq \ell_1(Q) := \min\limits_{1\leq i \leq j \leq n} Q_{ij}+ \frac{1}{\sum\limits_{k=1}^n \left(1/(Q_{kk} - \min\limits_{1\leq i \leq j \leq n} Q_{ij})\right)},
\]
where we define $1/0 = \infty$, $1/\infty = 0$, and $\beta + \infty = \infty$ for any $\beta \in \R$. These definitions imply that $\ell_1(Q) = \min_{1\leq i \leq j \leq n} Q_{ij} = \gamma_0$ if $\min\limits_{1\leq i \leq j \leq n} Q_{ij} = \min\limits_{k = 1,\ldots,n} Q_{kk}$.

\subsubsection{Lower Bound from Doubly Nonnegative Relaxation} \label{dnn_lb}

In this section, we present another lower bound on $\nu(Q)$ using an alternative formulation of (StQP).

A standard quadratic program can be equivalently reformulated as the following
 instance of a linear optimization problem over the convex cone of completely positive
 matrices~\cite{Bom_etal_00}:
\[
 \textrm{(CPP)} \quad \nu(Q) = \min\{\la Q, X \ra: \la E, X \ra =
 1, \quad X \in {\cal CP}^n\},
\]
 where $X \in \cS^n$ and $\cC \cP^n$ is given by \eqref{def_CPn}.
Despite the fact that (CPP) is a convex reformulation of (StQP), it remains NP-hard since the membership problem for the cone of completely positive matrices is intractable (see, e.g.,~\cite{DL14}).  

By \eqref{inclusions}, one can replace the intractable cone of completely positive matrices in (CPP) by the larger but tractable cone of doubly nonnegative matrices so as to obtain the 
following doubly nonnegative relaxation of (CPP):
\[
 \textrm{(DNN)} \quad \min\left\{\langle Q, X \rangle: \langle E, X \rangle = 1, \quad X \in {\cal DNN}^n \right\},
\]
where ${\cal DNN}^n$ is given by \eqref{def_Dn}.

Therefore, another lower bound on $\nu(Q)$ is given by 
\[
 \textrm{(LB2)} \quad \nu(Q) \geq \ell_2(Q) := \min\left\{\langle Q, X \rangle: \langle E, X \rangle = 1, \quad X \in \cD^n \right\}.
\]

By~\cite[Theorem 13]{BLT08}, we have the following relation: 
\begin{equation} \label{lb_comparison}
\ell_1(Q) \leq \ell_2(Q) \leq \nu(Q),
\end{equation}
i.e., the lower bound $\ell_2(Q)$ is at least as tight as $\ell_1(Q)$. By Lemma~\ref{lem1}(iv) and \eqref{lb_comparison}, both lower bounds are exact if the minimum entry of $Q$ lies along the diagonal. However, as illustrated by our computational results in Section~\ref{comp}, $\ell_2(Q)$ is, in general, much tighter than $\ell_1(Q)$.

Note that $\ell_1(Q)$ can be computed in $O(n^2)$ time whereas $\ell_2(Q)$ requires solving a computationally expensive semidefinite program. We remark that there exist other lower bounds in the literature (see, e.g.,~\cite{N99,BD02,AB05,BLT08}, and~\cite{BLT08} for a comparison). Usually, there is a trade-off between the quality of the lower bound and its computational cost.

\section{Mixed Integer Linear Programming Formulations} \label{milp_form}

In this section, we present two different mixed integer linear programming (MILP) reformulations of standard quadratic programs. We then propose a set of inequalities that are valid for both formulations.

\subsection{A Formulation Based on KKT Conditions} \label{kkt_milp}

Our first MILP formulation is obtained by exploiting the KKT conditions. We discuss how the nonlinear complementarity constraints can be linearized by employing binary variables. We also discuss how to obtain valid upper bounds for the big-$M$ parameters that arise from this linearization.

Given an instance of (StQP), if $x \in \Delta_n$ is a KKT point of (StQP), then $\nu(Q) = x^T Q x = \lambda$ by (\ref{eq1}), (\ref{eq2}), and (\ref{eq5}) (see also~\cite{GT73}). Based on this simple observation, (StQP) can be equivalently formulated as the following linear program with complementarity constraints:

 \begin{equation*}
  \begin{array}{llrcl}   
  \textrm{(LPCC1)} & \min & \lambda & & \\
   & & Q x - \lambda e - s & = & 0, \\
 & & e^T x & = & 1, \\
 & & x_j s_j & = & 0, \quad j = 1,\ldots,n, \\
 & & x & \geq & 0, \\
 & & s & \geq & 0. 
  \end{array}
  \end{equation*}

We can linearize the nonconvex complementarity constraints in (LPCC1) by using binary variables and big-$M$ constraints, which gives rise to the following MILP reformulation of (StQP):


 \begin{alignat}{10}
\label{obj1} & \textrm{(MILP1)} & \quad \min && \lambda & && \\
 \nonumber
 & &  \quad \textrm{s.t.} && & && \\
\label{c1}
   & & && Q x - \lambda e - s & = && \, \, 0, \\
\label{c2}   
 & & && e^T x & = && \, \, 1, \\
\label{c3}
 & & && x_j & \leq && \, \, y_j, \quad j = 1,\ldots,n, \\
 \label{c4}
 & & && s_j & \leq && \, \, M_j (1 - y_j),  \quad j = 1,\ldots,n, \\
\label{c5} 
 & & && x & \geq && \, \, 0, \\
\label{c6}
 & & && s & \geq && \, \, 0, \\
\label{c7} 
 & & && y_j & \in && \, \, \{0,1\}, \quad j = 1,\ldots,n.
  \end{alignat}

Note that, by \eqref{c3} and \eqref{c4}, the binary variable $y_j$ ensures that $x_j$ and $s_j$ cannot simultaneously be positive for each $j = 1,\ldots,n$. In particular, if $y_j = 1$, then $s_j = 0$ by (\ref{c4}) and $x_j$ is allowed to be positive. Since $x \in \Delta_n$, we have $0 \leq x_j \leq 1$, which implies that (\ref{c3}) yields a valid upper bound on $x_j$. On the other hand, if $y_j = 0$, we have $x_j = 0$ by \eqref{c3}, in which case we need valid upper bounds on the variable $s_j$ in (\ref{c4}). 

Next, we discuss how to obtain bounds for the big-$M$ parameters employed in (MILP1). By \eqref{c1}, 
 \begin{equation} \label{def_sj}
  s_j = e_j^T Q x - \lambda, \quad j = 1,\ldots,n.
 \end{equation}

We can obtain an upper bound on $s_j$ by deriving an upper bound for each of the terms on the right-hand side of \eqref{def_sj}. For the first term, since $x \in \Delta_n$, we have 
\begin{equation} \label{ub_ejQx}
e_j^T Q x = x^T Q e_j \leq \max_{i = 1,\ldots,n} Q_{ij}, \quad j = 1,\ldots,n.
\end{equation}
In order to bound the second term from above, any lower bound on $\lambda$ can be employed. Since $\lambda \geq \nu(Q)$ for any feasible solution of (MILP1), it follows that any lower bound on $\nu(Q)$ can be used to obtain an upper bound on $s_j,~j = 1,\ldots,n$. Indeed, let $\ell$ denote an arbitrary lower bound on $\nu(Q)$. For any feasible solution $(x,y,s,\lambda)$ of (MILP1), it follows from (\ref{def_sj}) and (\ref{ub_ejQx}) that 
\[
s_j = e_j^T Q x - \lambda \leq \max_{i = 1,\ldots,n} Q_{ij} - \nu(Q) \leq \max_{i = 1,\ldots,n} Q_{ij} - \ell, \quad j = 1,\ldots,n,
\]
which implies that 
\begin{equation} \label{def_Mj}
M_j = \max_{i = 1,\ldots,n} Q_{ij} - \ell, \quad j = 1,\ldots,n
\end{equation}
would be a valid choice in (MILP1). In particular, we use $\ell \in \{\ell_1(Q),\ell_2(Q)\}$ in our computational experiments.

\subsection{An Alternative Formulation} \label{alt_milp}

In this section, we present an alternative MILP formulation. Given an instance of (StQP), we first derive an underestimating function and an overestimating function for the objective function. We then establish useful relations of these two functions, which form the basis of our second formulation.

We start with a lemma that presents an underestimating and an overestimating function for the objective function of (StQP), both of which depend on the support of a feasible solution given by \eqref{def_Px}.

\begin{lemma} \label{sandwich}
For any $Q \in \cS^n$ and $x \in \Delta_n$, we have
\begin{equation} \label{sandw}
\min_{j \in {\cal P}(x)} e_j^T Q x \leq x^T Q x \leq \max_{j \in {\cal P}(x)} e_j^T Q x,
\end{equation}
where ${\cal P}(x)$, given by \eqref{def_Px}, denotes the support of $x$. Furthermore, if $x \in \Delta_n$ is a KKT point of (StQP), then 
\begin{equation} \label{exact_sandwich}
\min_{j \in {\cal P}(x)} e_j^T Q x = x^T Q x = \max_{j \in {\cal P}(x)} e_j^T Q x.
\end{equation}
\end{lemma}

\begin{proof}
For any $Q \in \cS^n$ and $x \in \Delta_n$,
\[
x^T Q x = \left( \sum_{j=1}^n x_j e_j^T\right) Qx = \sum_{j=1}^n x_j \left( e_j^T Q x \right) = \sum_{j \in  {\cal P}(x)} x_j \left(e_j^T Q x \right),
\]
i.e., $x^T Q x$ is a convex combination of $e_j^T Q x$ for $j \in {\cal P}(x)$, from which \eqref{sandw} follows.

Let $x \in \Delta_n$ be a KKT point of (StQP). Then, by (\ref{eq1}) -- (\ref{eq5}), 
\begin{eqnarray*} 
e_j^T Q x & = & \lambda, \quad j \in {\cal P}(x), \\
e_j^T Q x & \geq & \lambda, \quad j \in {\cal Z}(x).
\end{eqnarray*}
Furthermore $x^T Q x = \sum\limits_{j \in  {\cal P}(x)} x_j \left(e_j^T Q x \right) = \lambda \left(\sum\limits_{j \in  {\cal P}(x)} x_j\right) = \lambda$, which establishes (\ref{exact_sandwich}).
\end{proof}

%
%
%

Using Lemma~\ref{sandwich}, we next present an alternative characterization of $\nu(Q)$.

\begin{proposition} \label{diff_char}
Given an instance of (StQP),
\begin{equation} \label{alt_char}
\nu(Q) = \min_{x \in \Delta_n} \max_{j \in {\cal P}(x)} e_j^T Q x.
\end{equation}
\end{proposition}

\begin{proof}
For any $x \in \Delta_n$, we have $x^T Q x \leq \max\limits_{j \in {\cal P}(x)} e_j^T Q x$ by Lemma~\ref{sandwich}, which implies that $\nu(Q) \leq \min \limits_{x \in \Delta_n} \max \limits_{j \in {\cal P}(x)} e_j^T Q x$. Conversely, let $x^* \in \Delta_n$ be an optimal solution of (StQP). Then, $x^*$ is a KKT point, which implies that $\nu(Q) = (x^*)^T Q x^* = \max \limits_{j \in {\cal P}(x^*)} e_j^T Q x^*$ by Lemma~\ref{sandwich}, which establishes the reverse inequality. 

\end{proof}

We are now in a position to propose an alternative MILP formulation based on the characterization \eqref{alt_char} stated in Proposition~\ref{diff_char}.

 \begin{alignat}{10}
\label{obj2} &  \textrm{(MILP2)} & \quad \min && \alpha & && \\
 \nonumber
 & &  \quad \textrm{s.t.} && & && \\
\label{co1}  
   & &  && e_j^T Q x  & \leq && \, \, \alpha + z_j, \quad j = 1,\ldots,n,\\
\label{co2}   
 & & && e^T x & = && \, \, 1, \\
\label{co3} 
 & & && x_j & \leq && \, \, y_j, \quad j = 1,\ldots,n, \\
\label{co4} 
 & & && z_j & \leq && \, \, U_j (1 - y_j),  \quad j = 1,\ldots,n, \\
\label{co5} 
 & & && x & \geq && \, \, 0, \\
\label{co6} 
 & & && z & \geq && \, \, 0, \\
\label{co7}
 & & && y_j & \in && \, \, \{0,1\}, \quad j = 1,\ldots,n.
 \end{alignat}

Note that the auxiliary variable $\alpha$ is introduced and used in (\ref{obj2}) and (\ref{co1}) to linearize the maximum function on the right-hand side of \eqref{alt_char} and the binary variables $y_j$ are employed in (\ref{co4}) to ensure that the maximum is restricted only to the linear functions corresponding to the support of $x$ in \eqref{co1}. Indeed, if $j \in {\cal P}(x)$, then $x_j > 0$, which forces $y_j = 1$ by \eqref{co3} and $z_j = 0$ by \eqref{co4}. Otherwise, \eqref{co1} is a redundant constraint since $z_j$ can then take a positive value. Note that we again rely on big-$M$ parameters $U_j$ in (\ref{co4}). The next proposition presents a valid bound for these parameters.

\begin{proposition} \label{equiv_milp}
Given an instance of (StQP), (MILP2) is an equivalent reformulation of (StQP) if 
\[
U_j\geq M_j, \quad j = 1,\ldots,n,
\]
where $M_j$ is defined as in (\ref{def_Mj}) and $\ell$ is any lower bound on $\nu(Q)$.
\end{proposition}

\begin{proof}
Let $x \in \Delta_n$. Then, for each $j \in {\cal P}(x)$, we have $y_j = 1$ by \eqref{co3}, which implies that $z_j = 0$ by \eqref{co4} and $\alpha \geq \max\limits_{j \in {\cal P}(x)} e_j^T Q x \geq x^T Q x$ by (\ref{co1}) and by Lemma~\ref{sandwich}. Since the objective function minimizes $\alpha$, the best choice of $\alpha$ would be given by $\alpha = \max\limits_{j \in {\cal P}(x)} e_j^T Q x$. 

Consider an index $j \not \in {\cal P}(x)$. Let us define $z_j = \max\{0, e_j^T Q x - \alpha\} \geq 0$. Then, if $e_j^T Q x - \alpha < 0$, we have $z_j = 0 \leq M_j$, which satisfies the constraints of (MILP2). Otherwise,
\[
z_j = e_j^T Q x - \alpha \leq \max_{i = 1,\ldots,n} Q_{ij} - \alpha \leq \max_{i = 1,\ldots,n} Q_{ij} - x^T Q x \leq \max_{i = 1,\ldots,n} Q_{ij} - \nu(Q),
\]
which implies that $z_j \leq \max\limits_{i = 1,\ldots,n} Q_{ij} - \ell = M_j$, where $\ell$ is any lower bound on $\nu(Q)$. It follows that for each $x \in \Delta_n$, we can construct $y \in \R^n$, $z \in \R^n$, and $\alpha \in \R$ such that $\alpha = \max\limits_{j \in {\cal P}(x)} e_j^T Q x \geq x^T Q x$. By Lemma~\ref{sandwich}, if $x \in \Delta_n$ is a KKT point, then we can choose $\alpha = x^T Q x$. The equivalence follows. 

\end{proof}

We close this section by a brief comparison of (MILP1) and (MILP2). Note that the constraint set (\ref{co1}) in (MILP2) can be rewritten as 
\[
Qx - z - \alpha e \leq 0.
\]
Identifying the variables $z$ with $s$ and $\alpha$ with $\lambda$, a comparison with (\ref{c1}) in (MILP1) reveals that (MILP2) is in fact a relaxation of (MILP1).  It follows that, for any feasible solution $(x,y,s,\lambda)$ of (MILP1), we can define $z = s$ and $\alpha = \lambda$ so that $(x,y,z,\alpha)$ is a feasible solution of (MILP2). On the other hand, while each feasible solution $(x,y,s,\lambda)$ of (MILP1) corresponds to a KKT point of (StQP), we can construct a feasible solution $(x,y,z,\alpha)$ for {\em any} $x \in \Delta_n$ such that $\alpha \geq x^T Q x$, with equality if $x$ is a KKT point of (StQP). Therefore, (MILP2) can be viewed as an {\em exact relaxation} of (MILP1).

\subsection{Valid Inequalities}

In this section, we present a set of inequalities that are valid for both formulations (MILP1) and (MILP2).

Given an instance of (StQP), Theorem~\ref{clique_thm} presents a relation between the support of an optimal solution and the convexity graph of $Q$. This relation gives rise to the following theorem.

\begin{theorem}
The following inequalities are valid for both formulations (MILP1) and (MILP2):
\begin{equation} \label{valid_ineq}
y_i + y_j \leq 1, \quad \quad 1 \leq i < j \leq n ~~\textrm{s.t.}~~Q_{ii} + Q_ {jj} - 2 Q_{ij} \leq 0.
\end{equation}
\end{theorem}
\begin{proof}
In both formulations (MILP1) and (MILP2), the binary variables $y_j$ are equal to one if $j \in {\cal P}(x)$ for any feasible solution $x \in \Delta_n$. By Theorem~\ref{clique_thm}, there exists a global solution of (StQP) whose support set forms a stable set in the complement of the convexity graph $G = (V,E)$ of $Q$. The assertion follows.
\end{proof}

\section{Computational Results} \label{comp}

We report the results of our computational experiments in this section. We first describe the set of instances used in our experiments. Then, we explain our experimental setup in detail. Finally, we report performances of the proposed MILP formulations in comparison with several other global solution approaches from the literature.

\subsection{Set of Instances}

In an attempt to accurately assess the performances of the MILP formulations (MILP1) and (MILP2), we conducted extensive experiments on the following set of instances from the literature:

\begin{enumerate}

\item[(i)] {\em BLST Instances}~\cite{BLST16}: This set consists of 150 instances\footnote{Publicly available at {\tt http://or.dei.unibo.it/library/msc}} with $n = 30$ (BLST30) and 150 instances with $n = 50$ (BLST50). Each entry of $Q$ is randomly generated from a triangular distribution with parameters $a < c < b$, where $a$ and $c$ denote the minimum and maximum values and $c$ is the mode of the distribution.

\item[(ii)] {\em ST Instances}~\cite{ST08}: This set consists of 24 instances\footnote{Publicly available at {\tt http://www.iasi.cnr.it/{\textasciitilde}liuzzi/StQP/}} with $n = 100$ (ST100), 18 instances with $n = 200$ (ST200), 11 instances with $n = 500$ (ST500), and 1 instance with $n = 1000$ (ST1000). For each of these instances, the matrix $Q$ is randomly generated so that its convexity graph $G = (V,E)$ has a prespecified density. 

\item[(iii)] {\em DIMACS Instances}: It is well-known that the maximum stable set problem in graph theory can be formulated as an instance of (StQP)~\cite{MS65}. This set consists of (StQP) instances obtained from the complements of the 30 instances of the maximum clique problem\footnote{Publicly available at {\tt http://wwwdimacs.dimacs.rutgers.edu/pub/challenge/graph/benchmarks/clique/}} from the Second DIMACS Implementation Challenge with $n \in [28,300]$. These instances are divided into two groups based on the number of vertices. DIMACS1 consists of 8 instances with $n \in [28,171]$ and DIMACS2 is comprised of 22 instances with $n \in [200,300]$.

\item[(iv)] {\em BSU Instances}~\cite{BSU18}: This set consists of 20 ``hard'' instances with $n \in [5,24]$. Each of these instances is specifically constructed to harbor an exponential number of strict local minimizers. In particular, the number of strict local minimizers varies between $1.38^n$ and $1.49^n$.

\end{enumerate}

Note that Theorem~\ref{clique_thm} establishes a relation between the support of a global minimizer of an instance of (StQP) and the set of cliques of the associated convexity graph $G = (V,E)$. Denoting the density of the convexity graph by $\delta \in [0,1]$, it follows that the number of cliques in $G$ tends to increase as $\delta$ increases. Therefore, instances of (StQP) with larger values of $\delta$ contain a larger number of possible support sets for a global minimizer. Indeed, this difficulty is also reflected in earlier computational experiments (see, e.g.~\cite{ST08,LMP17}). It is also worth mentioning that the number of valid inequalities \eqref{valid_ineq} is given by $(1 - \delta) n (n-1)/2$. Therefore, for fixed $n$, the number of valid inequalities decreases as $\delta$ increases. For each set of instances, we therefore report the range of the parameter $\delta$. 

Recall that, for an instance of (StQP), if the minimum entry of $Q$ lies along the main diagonal, then $\nu(Q)$ equals that entry by Lemma~\ref{lem1}(iv). Therefore, we use this criterion as a preprocessing step in order to eliminate trivial instances. Apart from this, we do not use any other preprocessing procedure. After eliminating such trivial instances, we obtain a test bed that consists of a total of 376 instances. We summarize the statistics on the set of instances in Table~\ref{tab_ins_sum}.

\begin{table}[ht]
\begin{scriptsize}
\begin{center}
\begin{tabular}{|c||c|c|c|}
\hline
Instance Family & Number of Instances & $n$ & $\delta$ \\
\hline
\hline
BLST30 & 134 & 30 & [0.10,1.00] \\
BLST50 & 138 & 50 & [0.11,1.00] \\
\hline
ST100 & 24 & 100 & [0.25,0.90] \\
ST200 & 18 & 200 & [0.25,0.75] \\
ST500 & 11 & 500 & [0.25,0.50] \\
ST1000 & 1 & 1000 & 0.25 \\
\hline
DIMACS1 & 8 & [28,171] & [.35,.93] \\
DIMACS2 & 22 & [200,300] & [.08,.97] \\
\hline
BSU & 20 & [5,24] & [.6,1] \\
  \hline
 \end{tabular}
 \caption{Summary of Instances}
 \label{tab_ins_sum}
 \end{center}
 \end{scriptsize}
 \end{table}

As illustrated by Table~\ref{tab_ins_sum}, our test bed encompasses a large number of instances of (StQP) with varying sizes and characteristics. In particular, we note the higher density of the convexity graphs associated with the hard BSU instances. Indeed, 14 instances in this set have a density of 1; 5 instances have densities between 0.92 and 0.95; and only one instance has a density of 0.6, supporting our previous observation regarding the correlation between the difficulty of an instance and the density of the associated convexity graph.  
 
\subsection{Experimental Setup} 
 
Note that we propose two MILP formulations (MILP1) and (MILP2). For each formulation, one can use the lower bound $\ell_1(Q)$ or $\ell_2(Q)$. Finally, for each choice of the lower bound, we have the option of adding the valid inequalities \eqref{valid_ineq} or not, which implies that we have a total of 8 variants. For each variant with (MILP1), we add the following bound constraints on the variable $\lambda$:
\begin{equation} \label{bounds_lambda}
\ell_1(Q) \leq  \lambda \leq \min_{k = 1,\ldots,n} Q_{kk},
\end{equation}
where the upper bound follows from Lemma~\ref{lem1}(iv). Similarly, the corresponding constraints are added for each variant with (MILP2):
\begin{equation} \label{bounds_alpha}
\ell_2(Q) \leq \alpha \leq \min_{k = 1,\ldots,n} Q_{kk}.
\end{equation}

We compare the performances of our MILP formulations with three other global solution approaches, namely, the MILP formulation of~\cite{XVZ15}, which is publicly available at {\tt https://github.com/xiawei918/quadprogIP}, the quadratic programming (QP) solver of CPLEX, and the nonlinear programming (NLP) solver BARON. 

We solved all MILP formulations in MATLAB (version R2017b) using CPLEX (version 12.8.0) with the Cplex Class API provided in CPLEX for MATLAB Toolbox. Similarly, the QP solver of CPLEX was called in MATLAB using the same API. The NLP solver BARON (version 17.4.1) was called from GAMS (version 24.8.5) using the MATLAB interface. The computation of $\ell_2(Q)$ requires a semidefinite programming solver. For that purpose, we employed MOSEK (version 8.1.0.49) using the YALMIP interface (version R20180413)~\cite{Lof04}. 

We measured the running times in terms of wall clock time. In our experiments with CPLEX and BARON, we imposed a time limit of 3600 seconds for each MILP and QP problem. No time limit was imposed on MOSEK for the computation of $\ell_2(Q)$. Our computational experiments were carried out on a 64-bit HP workstation with 24 threads (2 sockets, 6 cores per socket, 2 threads per core) running Ubuntu Linux with 48 GB of RAM and Intel Xeon CPU E5-2667 processors with a clock speed of 2.90 GHz. In our experiments with CPLEX, we chose the deterministic parallel mode by setting {\tt cplex.parallelmode = 1} in order to have reproducible results. We remark that both CPLEX and MOSEK can take advantage of multiple threads. Similarly, we employed {\tt option threads = 24} in GAMS. However, we noticed that the wall clock time and the CPU time reported by BARON were virtually identical on all instances, suggesting that BARON did not take advantage of the multiple threads. Therefore, we caution the reader about the interpretation of the run times in our experiments with BARON.   

In CPLEX and BARON, we set the optimality gap tolerance to $10^{-6}$, which is given by
\begin{equation} \label{gap_def}
\frac{|{\tt bestbound} - {\tt bestsolution}|}{10^{-10} + |{\tt bestsolution}|}.
\end{equation}
The default settings were employed for all the other parameters of CPLEX, BARON, and MOSEK. 

We denote by MILP1-L1 and MILP1-L1-VI the MILP formulation (MILP1) using the lower bound $\ell_1(Q)$ with and without the set of valid inequalities \eqref{valid_ineq}, respectively. We replace the suffix L1 by L2 for the lower bound $\ell_2(Q)$. We use a similar convention for (MILP2). The MILP formulation of~\cite{XVZ15} is denoted by QP-IP, whereas CPLEX QP and BARON refer to the QP formulations solved by CPLEX and BARON, respectively. The doubly nonnegative relaxation (DNN) is denoted by DNN. For our MILP formulations with the lower bound $\ell_2(Q)$, the solution times exclude the computational effort for solving the corresponding doubly nonnegative relaxation (DNN) required to compute this lower bound, which is reported separately. Finally, for a solution $x \in \Delta_n$ reported by a solver, an index $j \in \{1,\ldots,n\}$ is considered to be in the support of $x$ (i.e., $j \in {\cal P}(x)$) if $x_j >  10^{-8}$.

\subsection{BLST Instances}

In this section, we report the performances on the BLST instances~\cite{BLST16}. Recall that BLST30 consists of 134 instances with $n = 30$ and BLST50 is comprised of 138 instances with $n = 50$. We report the summary of the results in Tables~\ref{blst30} and~\ref{blst50} for the BLST30 and BLST50 instances, respectively. In each table, we have a row for each of the eight variants using (MILP1) and (MILP2), and a row for each of QP-IP, CPLEX QP, BARON, and DNN. The total time taken by each approach (in seconds) over {\em all} instances in each set is reported in the second column. The third column reports the number of instances solved to optimality within the time limit of 3600 seconds. We present the number of instances on which the corresponding approach hits the time limit and the average optimality gap defined as in \eqref{gap_def} over these instances in the fourth and fifth columns, respectively. Recall that we do not impose any time limit for solving the DNN relaxation.


\begin{table}[!ht]
\begin{scriptsize}
\begin{center}
\begin{tabular}{|c||c||c||c|c|}
\hline
 & Total Time & OPT & Time Limit & Average Gap \\
\hline
\hline
MILP1-L1 & 26.97 & 134 & 0 & -- \\
MILP1-L1-VI & 30.71 & 134 & 0 & -- \\
MILP2-L1 & 30.50 & 134 & 0 & -- \\
MILP2-L1-VI & 33.80 & 134 & 0 & -- \\
\hline
MILP1-L2 & 34.67 & 134 & 0 & -- \\
MILP1-L2-VI & 24.10 & 134 & 0 & -- \\
MILP2-L2 & 25.65 & 134 & 0 & -- \\
MILP2-L2-VI & 29.28 & 134 & 0 & -- \\
\hline
QP-IP & 38.10 & 134 & 0 & --\\
CPLEX QP & 39370.66 & 124 & 10 & 6.14\% \\
BARON & 188005.68 & 96 & 38 & 115.91\% \\
\hline
DNN & 14.33 & 134 & -- & -- \\
  \hline
 \end{tabular}
 \caption{Performance on BLST30 (134 instances; $n = 30$)}
 \label{blst30}
 \end{center}
 \end{scriptsize}
 \end{table}

 
 \begin{table}[!h]
\begin{scriptsize}
\begin{center}
\begin{tabular}{|c||c||c||c|c|}
\hline
 & Total Time & OPT & Time Limit & Average Gap \\
\hline
\hline
MILP1-L1 & 50.30 & 138 & 0 & -- \\
MILP1-L1-VI & 50.46 & 138 & 0 & -- \\
MILP2-L1 & 49.91 & 138 & 0 & -- \\
MILP2-L1-VI & 55.54 & 138 & 0 & -- \\
\hline
MILP1-L2 & 45.06 & 138 & 0 & -- \\
MILP1-L2-VI & 41.32 & 138 & 0 & -- \\
MILP2-L2 & 30.48 & 138 & 0 & -- \\
MILP2-L2-VI & 35.25 & 138 & 0 & -- \\
\hline
QP-IP & 85.85 & 138 & 0 & --\\
CPLEX QP & 77072.98 & 124 & 14 & 13.74\% \\
BARON & 431110.51 & 19 & 95 & 246.14\% \\
\hline
DNN & 99.44 & 138 & -- & -- \\
  \hline
 \end{tabular}
 \caption{Performance on BLST50 (138 instances; $n = 50$)}
 \label{blst50}
 \end{center}
 \end{scriptsize}
 \end{table}

On each of the BLST30 and BLST50 instances, the lower bound $\ell_2(Q)$ (i.e., the optimal value of (DNN)) is either equal to $\nu(Q)$ or the difference between the two is at most $10^{-6}$, whereas the simple lower bound $\ell_1(Q)$ is always smaller than $\nu(Q)$. Therefore, $\ell_2(Q)$ is significantly tighter than $\ell_1(Q)$ over all instances. The support sizes of optimal solutions are in the range $[1,10]$.

As illustrated by Tables~\ref{blst30} and~\ref{blst50}, our MILP formulations as well as QP-IP can solve each instance to optimality in a fraction of a second on average. On the other hand, each of CPLEX-QP and BARON hits the time limit on some of the instances, with CPLEX QP exhibiting a better performance than BARON. Note that the average gaps of BARON are usually considerably higher than those reported by CPLEX-QP. While the total computational effort for solving the DNN relaxation is quite negligible on BLST30, it increases significantly on BLST50, even surpassing the total time required for solving each MILP formulation. On each set, the total time taken by each of CPLEX QP and BARON is significantly larger than that required for each MILP formulation.

Each of the eight variants of (MILP1) and (MILP2) slightly outperforms QP-IP on BLST30. The improvement is more significant on BLST50. Furthermore, while the performance of each of the eight variants is similar on BLST30, (MILP2) seems to outperform (MILP1) on BLST50, especially when used with the lower bound $\ell_2(Q)$. Note that the valid inequalities do not seem to have a significant effect on these instances. Finally, we remark that the total time of each of our MILP formulations on BLST50 increases only modestly in comparison with BLST30. 
 
\subsection{ST Instances}

In this section, we report our computational results on ST instances~\cite{ST08}, which consists of 24 instances with $n = 100$ (ST100), 18 instances with $n = 200$ (ST200), 11 instances with $n = 500$ (ST500), and one instance with $n = 1000$ (ST1000). We report our results in Tables~\ref{st100},~\ref{st200},~\ref{st500}, and~\ref{st1000} for ST100, ST200, ST500, and ST1000, respectively, each of which is organized similarly to Table~\ref{blst30}. Note that the DNN relaxation could not be solved for any instance in ST500 and ST1000. We therefore omit the rows corresponding to the variants of our MILP formulations with the lower bound $\ell_2(Q)$ and the DNN relaxation in Tables~\ref{st500} and~\ref{st1000} accordingly. 


\begin{table}[!h]
\begin{scriptsize}
\begin{center}
\begin{tabular}{|c||c||c||c|c|}
\hline
 & Total Time & OPT & Time Limit & Average Gap \\
\hline
\hline
MILP1-L1 & 38.21 & 24 & 0 & -- \\
MILP1-L1-VI & 52.95 & 24 & 0 & -- \\
MILP2-L1 & 20.78 & 24 & 0 & -- \\
MILP2-L1-VI & 28.16 & 24 & 0 & -- \\
\hline
MILP1-L2 & 44.22 & 24 & 0 & -- \\
MILP1-L2-VI & 63.88 & 24 & 0 & -- \\
MILP2-L2 & 10.84 & 24 & 0 & -- \\
MILP2-L2-VI & 16.87 & 24 & 0 & -- \\
\hline
QP-IP & 135.01 & 24 & 0 & --\\
CPLEX QP & 70894.32 & 6 & 18 & 38.18\% \\
BARON & 86404.50 & 0 & 24 & 1131.58\% \\
\hline
DNN & 388.39 & 24 & -- & -- \\
  \hline
 \end{tabular}
 \caption{Performance on ST100 (24 instances; $n = 100$)}
 \label{st100}
 \end{center}
 \end{scriptsize}
 \end{table}


\begin{table}[!h]
\begin{scriptsize}
\begin{center}
\begin{tabular}{|c||c||c||c|c|}
\hline
 & Total Time & OPT & Time Limit & Average Gap \\
\hline
\hline
MILP1-L1 & 576.13 & 18 & 0 & -- \\
MILP1-L1-VI & 512.49 & 18 & 0 & -- \\
MILP2-L1 & 189.45 & 18 & 0 & -- \\
MILP2-L1-VI & 226.08 & 18 & 0 & -- \\
\hline
MILP1-L2 & 463.57 & 18 & 0 & -- \\
MILP1-L2-VI & 417.84 & 18 & 0 & -- \\
MILP2-L2 & 37.96 & 18 & 0 & -- \\
MILP2-L2-VI & 124.83 & 18 & 0 & -- \\
\hline
QP-IP & 1197.43 & 18 & 0 & --\\
CPLEX QP & 64804.65 & 0 & 18 & 44.63\% \\
BARON & 64807.69 & 0 & 18 & 4249.18\% \\
\hline
DNN & 7729.80 & 18 & -- & -- \\
  \hline
 \end{tabular}
 \caption{Performance on ST200 (18 instances; $n = 200$)}
 \label{st200}
 \end{center}
 \end{scriptsize}
 \end{table}

We first focus on the instances in ST100 and ST200. On each of these instances, the difference between the lower bound $\ell_2(Q)$ and $\nu(Q)$ is less than $10^{-5}$. On the other hand, the simple lower bound $\ell_1(Q)$ is considerably smaller than $\nu(Q)$ on all instances. The support sizes of optimal solutions vary between 3 and 7.

A close examination of Tables~\ref{st100} and~\ref{st200} reveals that MILP formulations significantly outperform each of CPLEX QP and BARON on ST100 and ST200. In particular, all of the instances in these two sets can be solved to optimality by each of our MILP formulations and by QP-IP, whereas CPLEX QP can solve only 6 instances out of 24 in ST100 to optimality and BARON can solve none within the time limit. Furthermore, BARON generally reports considerably higher optimality gaps in comparison with CPLEX QP. Note again that the total computational effort required for each MILP formulation is significantly smaller than that for CPLEX QP and BARON. The total computational effort for solving the DNN relaxation similarly exceeds the total effort required for each MILP formulation. In particular, the average solution time of DNN increases from 16.22 seconds on ST100 to 429.43 seconds on ST200, an increase of more than 25-fold despite the fact that $n$ only increases by a factor of 2.

Each of the eight variants of our MILP formulations outperforms QP-IP on both ST100 and ST200. MILP2-L2 stands out as a clear winner among the other MILP formulations. Note, in particular, that MILP-L2 is about 13 times faster than QP-IP on ST100 and more than 30 times faster on ST200. The lower bound $\ell_2(Q)$ seems to have a significantly positive effect on the performance of (MILP2), both with and without valid inequalities and, to a lesser extent, on the performance of (MILP1) on ST200. The inclusion of valid inequalities has a mixed effect on (MILP1) and (MILP2). It is worth mentioning that instances with sparse convexity graphs give rise to a larger number of valid inequalities, which may adversely affect the overall performance of the variants with valid inequalities.


\begin{table}[!h]
\begin{scriptsize}
\begin{center}
\begin{tabular}{|c||c||c||c|c|}
\hline
 & Total Time & OPT & Time Limit & Average Gap \\
\hline
\hline
MILP1-L1 & 6583.05 & 11 & 0 & -- \\
MILP1-L1-VI & 8960.10 & 11 & 0 & -- \\
MILP2-L1 & 2519.94 & 11 & 0 & -- \\
MILP2-L1-VI & 3802.81 & 11 & 0 & -- \\
\hline
QP-IP & 20291.34 & 6 & 5 & 6.97\% \\
CPLEX QP & 39603.91 & 0 & 11 & 44660285.88\% \\
BARON & 39600.13 & 0 & 11 & 11376.00\% \\
  \hline
 \end{tabular}
 \caption{Performance on ST500 (11 instances; $n = 500$)}
 \label{st500}
 \end{center}
 \end{scriptsize}
 \end{table}


\begin{table}[!h]
\begin{scriptsize}
\begin{center}
\begin{tabular}{|c||c||c||c|c|}
\hline
 & Total Time & OPT & Time Limit & Average Gap \\
\hline
\hline
MILP1-L1 & 3600.31 & 0 & 1 & 7.63\% \\
MILP1-L1-VI & 3609.65 & 0 & 1 & 46.74\% \\
MILP2-L1 & 2100.22 & 1 & 0 & -- \\
MILP2-L1-VI & 3600.18 & 0 & 1 & 11.59\% \\
\hline
QP-IP & 3607.88 & 0 & 1 & 22.05\% \\
CPLEX QP & 3600.95 & 0 & 1 & 86423971.55\% \\
BARON & 3601.03 & 0 & 1 & 18020.23\% \\
  \hline
 \end{tabular}
 \caption{Performance on ST1000 (1 instance; $n = 1000$)}
 \label{st1000}
 \end{center}
 \end{scriptsize}
 \end{table}

We now focus on the larger instances ST500 and ST1000. On each of these instances, $\ell_1(Q)$ again turns out to be a rather loose lower bound on $\nu(Q)$. The sizes of the support of optimal solutions are between 4 and 6. 

Tables~\ref{st500} and~\ref{st1000} report the results on ST500 and ST1000, respectively. Recall that the DNN relaxation cannot be solved on any of these instances. Therefore, we present the results only for the variants of our MILP formulations with the simple lower bound $\ell_1(Q)$. On ST500, while each of our formulations can solve all instances to optimality, QP-IP hits the time limit on 5 of the 11 instances. Furthermore, both CPLEX QP and BARON also hit the time limit on all of the instances. ST1000 consists of a single instance with $n = 1000$. We include this instance in our experiments to assess the performances of different global approaches on a very large sample instance. This instance indeed turns out to be very challenging for each approach except MILP2-L1, which manages to solve it to optimality within the time limit. We believe that this instance provides a remarkable computational evidence about the effectiveness of (MILP2), even when used with the simple and fairly loose lower bound $\ell_1(Q)$. We also point out the extremely large optimality gaps reported by the QP solvers, which illustrates that instances of this scale are well beyond the reach of current state-of-the-art QP and NLP solvers.

On ST500, a comparison of total computational effort indicates that our MILP formulations are far more effective than the other approaches. On ST1000, only MILP1-L1-VI reports a larger gap than that of QP-IP. Note that the valid inequalities do not seem to help on these two sets since the total number of such inequalities can be fairly large. For instance, the number of valid inequalities is about 375000 on the single instance in ST1000. On both of these sets, MILP2-L1 clearly outperforms all the other MILP formulations.

\subsection{DIMACS Instances}

In this section, we report our computational results on DIMACS instances, consisting of 8 instances in DIMACS1 with $n \in [28,171]$, and 22 instances in DIMACS2 with $n \in [200,300]$. Each of these instances is obtained from the reformulation of the maximum stable set problem as an instance of (StQP)~\cite{MS65}. We first review this formulation. 

Let $G = (V,E)$ be a simple, undirected graph, where $V = \{1,\ldots,n\}$. Recall that a set $S \subseteq V$ is a stable set if no two nodes in $S$ are connected by an edge. The maximum stable set problem is concerned with computing the largest stable set in $G$, whose size is denoted by $\alpha(G)$. 

By defining a binary variable $y_j \in \{0,1\}$ for each $j \in V$ to indicate whether node $j$ belongs to a maximum stable set, the maximum stable set problem can be formulated as an integer linear programming problem as follows:
\[
\textrm{(ILP)} \quad \alpha(G) = \max\left\{\sum\limits_{j=1}^n y_j: y_i + y_j \leq 1, \quad (i,j) \in E, \quad y_j \in \{0,1\}, \quad j = 1,\ldots,n\right\}.
\] 

Motzkin and Straus~\cite{MS65} proposed the following formulation in the form of (StQP):
\[
\textrm{(MS-StQP)} \quad \frac{1}{\alpha(G)} = \min\limits_{x \in \Delta_n} x^T(I + A_G)x,
\]
where $A_G \in {\cal S}^n$ denotes the adjacency matrix of $G$. Furthermore, if $S^* \subseteq V$ denotes a maximum stable set of $G$, then an optimal solution of (MS-StQP) is given by $x_j = 1/|S^*|$ if $j \in S^*$, and $x_j = 0$ otherwise. 

While we think that reformulating a combinatorial optimization problem as an instance of (StQP) and then reformulating it as an MILP problem may not necessarily be the best solution approach, we still include this set of instances in our experiments due to the existence of a global optimal solution of (MS-StQP) with a support of size $\alpha(G)$, which can be fairly large for certain classes of graphs. Therefore, these instances can be particularly challenging for MILP formulations of (StQP). 

Given an instance of (MS-StQP) corresponding to a graph $G = (V,E)$, it is easy to verify that the convexity graph of $Q = I + A_G$ (see Section~\ref{prop_opt_sol}) is precisely given by the complement of $G$. Since the valid inequalities \eqref{valid_ineq} are defined for each edge of the complement of the convexity graph, which coincides with $G$, it follows that the valid inequalities in our MILP formulations are given by 
\[
y_i + y_j \leq 1, \quad (i,j) \in E,
\]
which are precisely the same constraints as in (ILP). Therefore, our MILP formulations with valid inequalities can in some sense be viewed as extended formulations of (ILP).

We report our computational results on DIMACS1 and DIMACS2 in Tables~\ref{dimacs1} and~\ref{dimacs2}, respectively. For comparison purposes, we also include the performance of the integer linear programming formulation (ILP) on these instances, denoted by ILP. 


\begin{table}[!ht]
\begin{scriptsize}
\begin{center}
\begin{tabular}{|c||c||c||c|c|}
\hline
 & Total Time & OPT & Time Limit & Average Gap \\
\hline
\hline
MILP1-L1 & 3697.94 & 7 & 1 & 72.80\% \\
MILP1-L1-VI & 52.31 & 8 & 0 & -- \\
MILP2-L1 & 3624.48 & 7 & 1 & 72.80\% \\
MILP2-L1-VI & 59.72 & 8 & 0 & -- \\
\hline
MILP1-L2 & 3731.06 & 7 & 1 & 9.46\% \\
MILP1-L2-VI & 140.79 & 8 & 0 & -- \\
MILP2-L2 & 3616.89 & 7 & 1 & 9.46\% \\
MILP2-L2-VI & 51.82 & 8 & 0 & -- \\
\hline
QP-IP & 4863.33 & 7 & 1 & 100.00\% \\
CPLEX QP & 29449.29 & 0 & 8 & 61.79\% \\
BARON & 28804.96 & 0 & 8 & 83.13\% \\
ILP & 16.17 & 8 & 0 & -- \\
\hline
DNN & 356.74 & 8 & -- & -- \\
  \hline
 \end{tabular}
 \caption{Performance on DIMACS1 (8 instances; $n \in [28,171]$)}
 \label{dimacs1}
 \end{center}
 \end{scriptsize}
 \end{table}

 
 \begin{table}[ht]
\begin{scriptsize}
\begin{center}
\begin{tabular}{|c||c||c||c|c|}
\hline
 & Total Time & OPT & Time Limit & Average Gap \\
\hline
\hline
MILP1-L1 & 57676.31 & 7 & 15 & 81.91\% \\
MILP1-L1-VI & 21463.40 & 18 & 4 & 71.60\% \\
MILP2-L1 & 55529.78 & 7 & 15 & 82.38\% \\
MILP2-L1-VI & 25214.24 & 17 & 5 & 79.65\% \\
\hline
MILP1-L2 & 56047.95 & 7 & 15 & 18.29\% \\
MILP1-L2-VI & 27944.05 & 17 & 5 & 9.40\% \\
MILP2-L2 & 55127.62 & 8 & 14 & 20.60\% \\
MILP2-L2-VI & 30903.19 & 15 & 7 & 7.65\% \\
\hline
QP-IP & 61884.01 & 5 & 17 & 99.98\% \\
CPLEX QP & 77391.14 & 1 & 21 & 84.68\% \\
BARON & 75616.99 & 1 & 21 & 92.79\% \\
ILP & 4672.59 & 21 & 1 & 14.07\% \\
\hline
DNN & 25071.78 & 22 & -- & -- \\
  \hline
 \end{tabular}
 \caption{Performance on DIMACS2 (22 instances; $n \in [200,300]$)}
 \label{dimacs2}
 \end{center}
 \end{scriptsize}
 \end{table}

In contrast with BLST instances and ST instances, the lower bound $\ell_2(Q)$ is almost tight on 5 out of 8 instances in DIMACS1 and only on 6 out of 22 instances in DIMACS2. The lower bound $\ell_1(Q)$ is quite loose across all instances. The support size of an optimal solution is in the range of [4,44] and [8,128] in DIMACS1 and DIMACS2, respectively.

Tables~\ref{dimacs1} and~\ref{dimacs2} illustrate that these instances are indeed challenging for each of the global solution approaches. We discuss the computational results on DIMACS1 and DIMACS2 separately.

On DIMACS1, as illustrated by Table~\ref{dimacs1}, each variant of our MILP formulations without valid inequalities can solve 7 of the 8 instances to optimality within the time limit. The addition of the valid inequalities not only helps to close the gap on the remaining instance but also significantly improves the overall performance of both (MILP1) and (MILP2) regardless of the particular lower bound employed. Similarly, QP-IP can solve 7 out of 8 instances to optimality. Each of CPLEX QP and BARON hits the time limit on each of the 8 instances, with BARON reporting a higher average optimality gap compared to CPLEX QP. The total computational effort for solving the DNN relaxation is larger than the total solution time for each variant of our MILP formulations with valid inequalities. Finally, ILP can solve all of these instances to optimality fairly quickly. We conjecture that this favorable outcome can be attributed to CPLEX's capability of identifying the particular structure of this formulation and adding other well-known inequalities and cuts. 

Overall, each of our eight MILP variants requires less computational effort than QP-IP, either achieving significantly smaller solution times or terminating with smaller optimality gaps. While the performances of MILP1-L1-VI and MILP2-L1-VI are similar, MILP2-L2-VI exhibits a significantly better performance than MILP1-L2-VI. 

On DIMACS2, which consists of larger instances, Table~\ref{dimacs2} reveals that each approach, including ILP, hits the time limit on various subsets of instances, illustrating the challenging structure of these instances. Similar to DIMACS1, ILP achieves the best performance both in terms of the total time and the number of instances solved to optimality. In particular, ILP is terminated due to the time limit only on one instance out of 22. MILP1-L1-VI ranks in the second place, with the second smallest total time and the second largest number of instances solved to optimality, followed by MILP2-L1-VI and MILP1-L2-VI, respectively. Once again, QP-IP is outperformed by each of our eight MILP variants in terms of each of total time, number of instances solved to optimality, and average optimality gap. Both CPLEX QP and BARON can solve only one instance to optimality and hit the time limit on each of the remaining instances, with CPLEX QP reporting a slightly lower average gap compared to BARON. The total computational effort required for solving the DNN relaxation is quite significant, exceeding, for instance, the total solution time of MILP1-L1-VI.

We again observe significant improvements due to the addition of valid inequalities for both (MILP1) and (MILP2). Furthermore, it is worth noting that both formulations (MILP1) and (MILP2) significantly benefit in terms of average optimality gaps when used with the better lower bound $\ell_2(Q)$. 

\subsection{BSU Instances}

In this section, we report our computational results on the BSU instances. This set consists of one instance for each value of $n$ in the range $[5,24]$ and is specifically constructed to have an exponential number of strict local minimizers. We report our results in Table~\ref{bsu}, which is organized similarly to Table~\ref{blst30}.


\begin{table}[ht]
\begin{scriptsize}
\begin{center}
\begin{tabular}{|c||c||c|c||c|c||}
\hline
 & Total Time & OPT & Time Limit & Average Gap \\
\hline
\hline
MILP1-L1 & 6.75 & 20 & 0 & -- \\
MILP1-L1-VI & 5.99 & 20 & 0 & -- \\
MILP2-L1 & 8.77 & 20 & 0 & -- \\
MILP2-L1-VI & 7.39 & 20 & 0 & -- \\
\hline
MILP1-L2 & 7.02 & 20 & 0 & -- \\
MILP1-L2-VI & 4.92 & 20 & 0 & -- \\
MILP2-L2 & 7.04 & 20 & 0 & -- \\
MILP2-L2-VI & 7.45 & 20 & 0 & -- \\
\hline
QP-IP & 11.96 & 20 & 0 & --\\
CPLEX QP & 63194.15 & 5 & 15 & 14.82\% \\
BARON & 58022.22 & 5 & 15 & 54.53\% \\
\hline
DNN & 0.51 & 20 & -- & -- \\
  \hline
 \end{tabular}
 \caption{Performance on BSU (20 instances; $n \in [5,24]$)}
 \label{bsu}
 \end{center}
\end{scriptsize}
 \end{table}

The lower bound $\ell_2(Q)$ is tight only on one instance in this set, namely for $n = 6$. On the other hand, the lower bound $\ell_1(Q)$ is considerably weaker than $\ell_2(Q)$ on all instances. The sizes of the support of an optimal solution range between 2 and 13.

Table~\ref{bsu} reveals that each instance in this set can be solved by each MILP formulation in a fraction of a second on average, illustrating that the existence of an exponential number of strict local minimizers does not seem to hinder the performances of MILP formulations. In particular, each of our eight MILP variants slightly outperforms QP-IP on this set. On the other hand, each of CPLEX QP and BARON can solve only the five smallest instances to optimality within the time limit, each hitting the time limit on the remaining set of 15 instances. Therefore, these instances seem to be particularly challenging for QP and NLP solvers despite the small values of $n$. Note that DNN requires a negligible total computational effort on this set. 

This set of instances provides strong computational evidence that MILP formulations are less likely to be affected by the possibility of an exponential number of local minimizers, which, in general, poses a great challenge for general purpose QP and NLP solvers. These results illustrate that MILP formulations can be much more effective for solving standard quadratic programs in comparison with general purpose QP and NLP solvers.

\subsection{Overall Comparison}

In this section, we present an overall comparison of each of the global solution approaches on all of the instances in our test bed. In an attempt to make a fair comparison, we first divide the set of instances into two subsets. IS1 consists of all instances on which each global solution approach is attempted, i.e., the set of all instances excluding the sets ST500 and ST1000. We collect these two sets of larger instances in the set IS2. Recall that the DNN relaxation cannot be solved on IS2. 

Our first comparison is based on the summary of the performances of all of the global solution approaches on our test bed. In Table~\ref{overall}, we report the total time required by each approach as well as the number of instances solved to optimality and the number of instances terminated due to the time limit. We organize the results similarly as in Table~\ref{blst30}, except that we have a separate set of columns for each of IS1 and IS2. Furthermore, we do not include the average optimality gaps due to the high variability. 

\begin{table}[ht]
\begin{scriptsize}
\begin{center}
\begin{tabular}{|c||c|c|c||c|c|c|}
\hline
 & \multicolumn{3}{|c||}{IS1} & \multicolumn{3}{|c|}{IS2} \\
\hline
 & Total Time & OPT & Time Limit & Total Time & OPT & Time Limit  \\
\hline
\hline
MILP1-L1 & 62072.61 & 348 & 16 & 10183.36 & 11 & 1 \\
MILP1-L1-VI & 22168.31 & 360 & 4 &  12569.75 & 11 & 1 \\
MILP2-L1 & 59453.67 & 348 & 16 & 4620.16 & 12 & 0\\
MILP2-L1-VI & 25624.93 & 359 & 5 & 7402.99 & 11 & 1 \\
\hline
MILP1-L2 & 60373.55 & 348 & 16 &  -- & -- & --\\
MILP1-L2-VI & 28636.90 & 358 & 5 & -- & -- & -- \\
MILP2-L2 & 58846.48 & 349 & 15 & -- & -- & --\\
MILP2-L2-VI & 31168.69 & 357 & 7 & -- & -- & -- \\
\hline
QP-IP & 68215.69 & 346 & 18 & 23899.22 & 6 & 6 \\
CPLEX QP & 422177.19 & 260 & 104 & 43204.86 & 0 & 12 \\
BARON & 932772.55 & 121 & 243 &  43201.16 & 0 & 12 \\
\hline
DNN & 33660.99 & 364 & -- & -- & -- & -- \\
  \hline
 \end{tabular}
 \caption{Overall performance comparison (IS1: 364 instances; $n \in [5,300]$; IS2: 12 instances; $n \in [500,1000]$)}
 \label{overall}
 \end{center}
\end{scriptsize}
 \end{table}

We first focus on instances in IS1. Table~\ref{overall} reveals that MILP1-L1-VI dominates all the other approaches both in terms of the total time and the number of instances solved to optimality. MILP2-L1-VI ranks in the second place, followed closely by MILP1-L2-VI and MILP2-L2-VI, respectively. Each of our 8 MILP variants outperforms QP-IP on both performance metrics. CPLEX QP and BARON fall behind significantly, with CPLEX QP exhibiting better performance than BARON. The overall computational effort required for solving the DNN relaxation exceeds the total time required by each of MILP1-L1-VI, MILP2-L1-VI, MILP1-L2-VI, and MILP2-L2-VI. Note that the addition of valid inequalities improves both performance metrics on both formulations regardless of the lower bound employed. 

On IS1, while the use of the tighter lower bound $\ell_2(Q)$ leads to a slight improvement over $\ell_1(Q)$ using both formulations (MILP1) and (MILP2) without the valid inequalities, their inclusion seems to adversely affect the overall performance. We remark, however, that the addition of these inequalities almost always leads to smaller optimality gaps on instances terminated due to the time limit. Therefore, valid inequalities may potentially help to close the optimality gap faster if a larger time limit is imposed.

Note that IS2 consists of 12 larger instances. On this set, MILP2-L1 seems to outperform all other approaches in terms of both performance metrics, which is followed, in turn, by MILP2-L1-VI, MILP1-L1, and MILP1-LI-VI, respectively. Once again, each of our MILP variants exhibits better performance than QP-IP on both metrics. Each of CPLEX QP and BARON hits the time limit on each instance in this set. Note that valid inequalities seem to degrade the performance, possibly due to the large number of such inequalities. 

In an attempt to shed more light on the comparison of the performances of different approaches, we report performance profiles~\cite{DM02}, which are frequently used for benchmarking purposes in optimization. For a given problem instance, the performance ratio of a particular approach is defined as the ratio of the solution time of that approach to the best solution time among all approaches on that particular instance, which is defined to be $+\infty$ if the approach fails to solve the instance. Then, for each approach $a$, a cumulative distribution function $P_a(\tau)$ is defined to be the percentage of the number of instances that can be solved by that approach within a factor of $\tau$ of the solution time of the best approach. Note that only the instances that can be solved by at least one approach within the time limit are included in the comparison. 

\begin{figure}[!h]
\begin{center}
\includegraphics[scale=0.5]{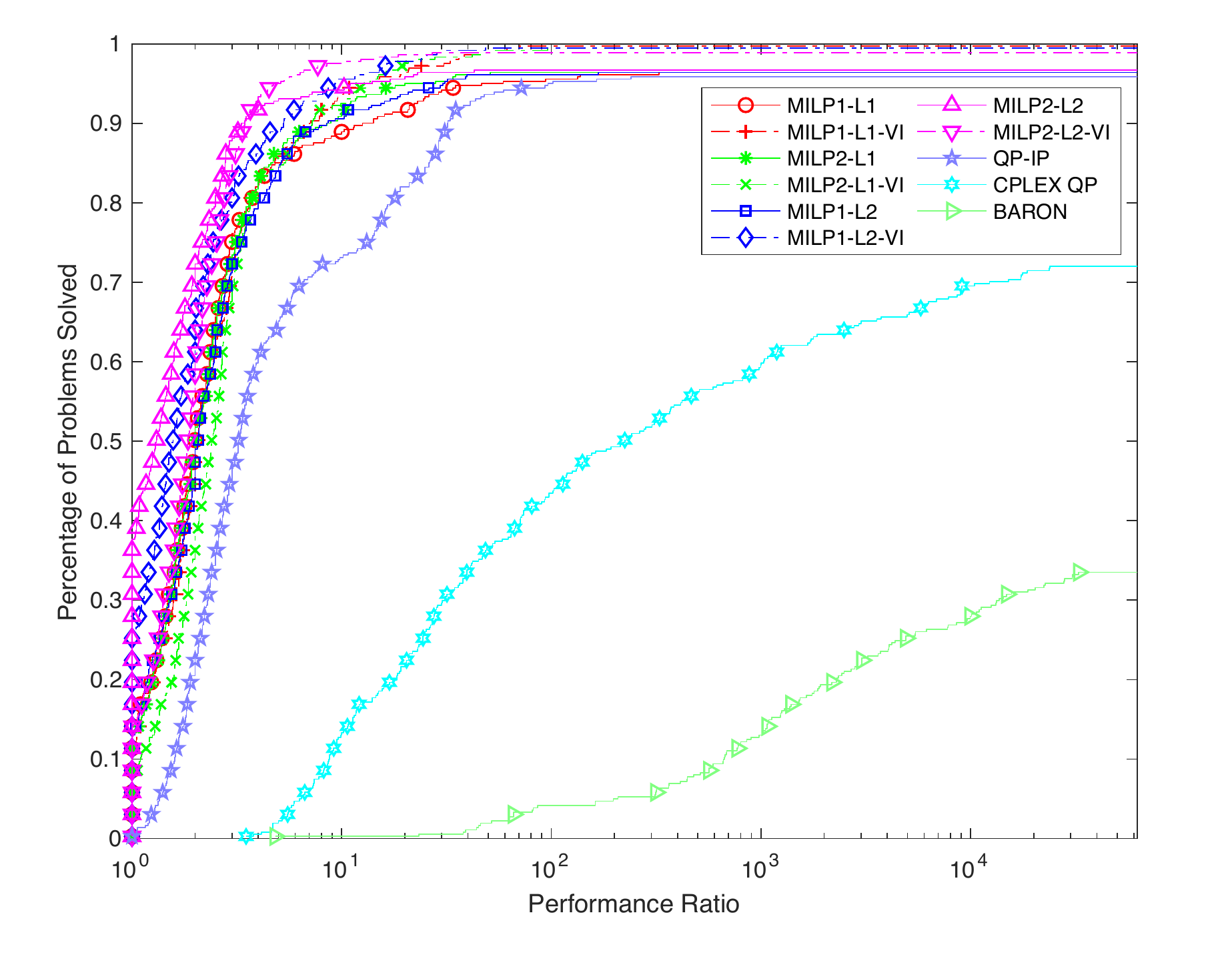}
\end{center}
\caption{Performance Profile on IS1}
\label{fig1}
\end{figure}

\begin{figure}[!b]
\begin{center}
\includegraphics[scale=0.5]{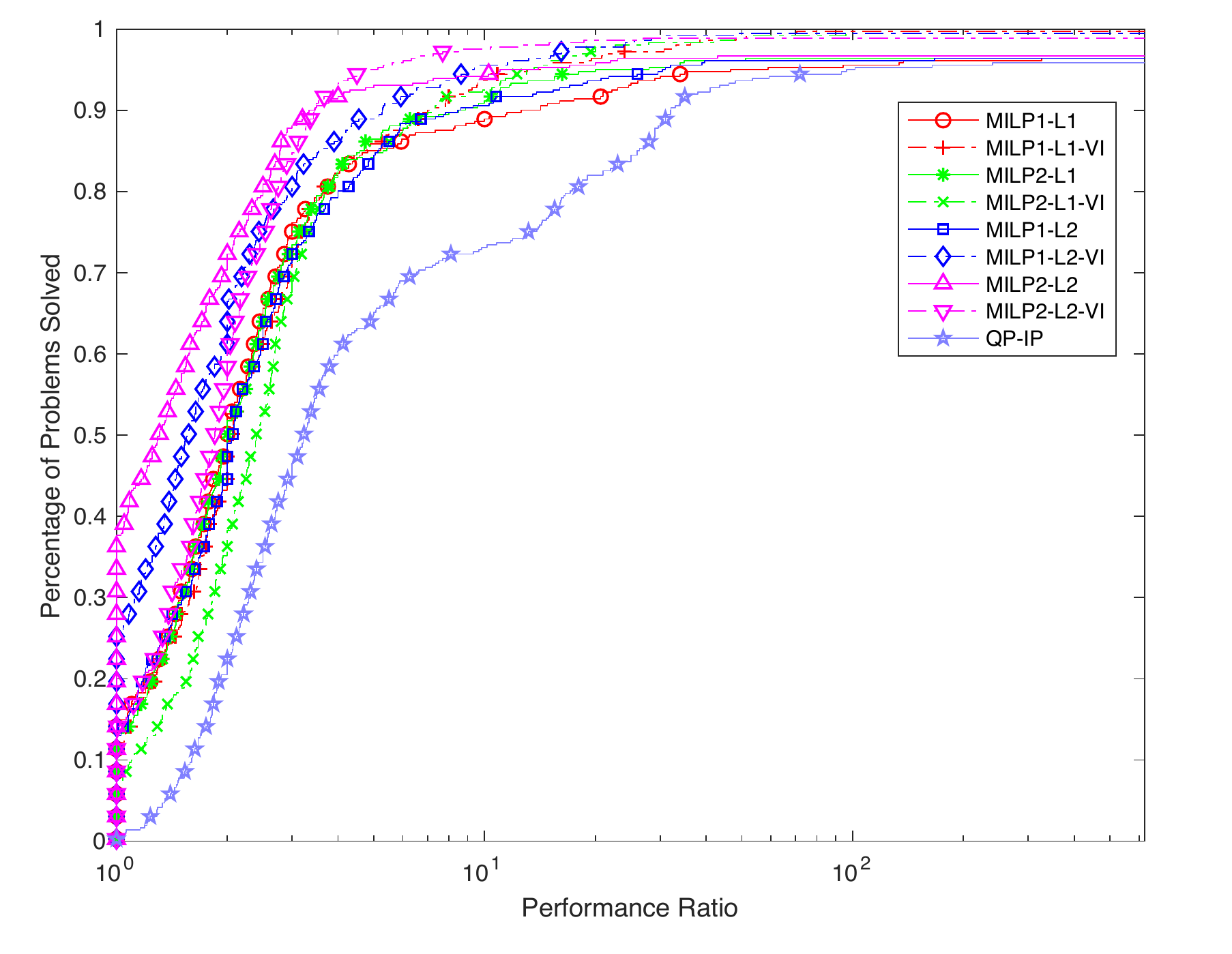}
\end{center}
\caption{Performance Profile on IS1 (excluding CPLEX QP and BARON)}
\label{fig2}
\end{figure}

In Figures~\ref{fig1} and~\ref{fig2}, we report two performance profiles for the instance set IS1. Out of 364 instances in this set, the figures compare the performance ratios on 361 of them, which can be solved by at least one approach. Figure~\ref{fig1} includes all of the approaches whereas CPLEX QP and BARON are excluded in Figure~\ref{fig2} to facilitate an easier comparison among MILP formulations only. Note that performance ratios are reported in logarithmic scale. Furthermore, the markers are included for every 20th performance ratio. 

Both Figures~\ref{fig1} and~\ref{fig2} indicate the superior performance of the proposed MILP formulations over the other global solution approaches on the instance set IS1. In particular, it is worth noticing that each of our 8 MILP variants outperforms QP-IP, which, in turn, exhibits better performance than CPLEX-QP, followed by BARON. In terms of the best solution time, MILP2-L2 exhibits the best performance, followed by MILP1-L2-VI, and MILP2-L2-VI, respectively. Note, however, that each of MILP2-L2-VI and MILP1-L2-VI manages to solve a higher percentage of instances within the time limit than MILP2-L2. 

\begin{figure}[!h]
\begin{center}
\includegraphics[scale=0.5]{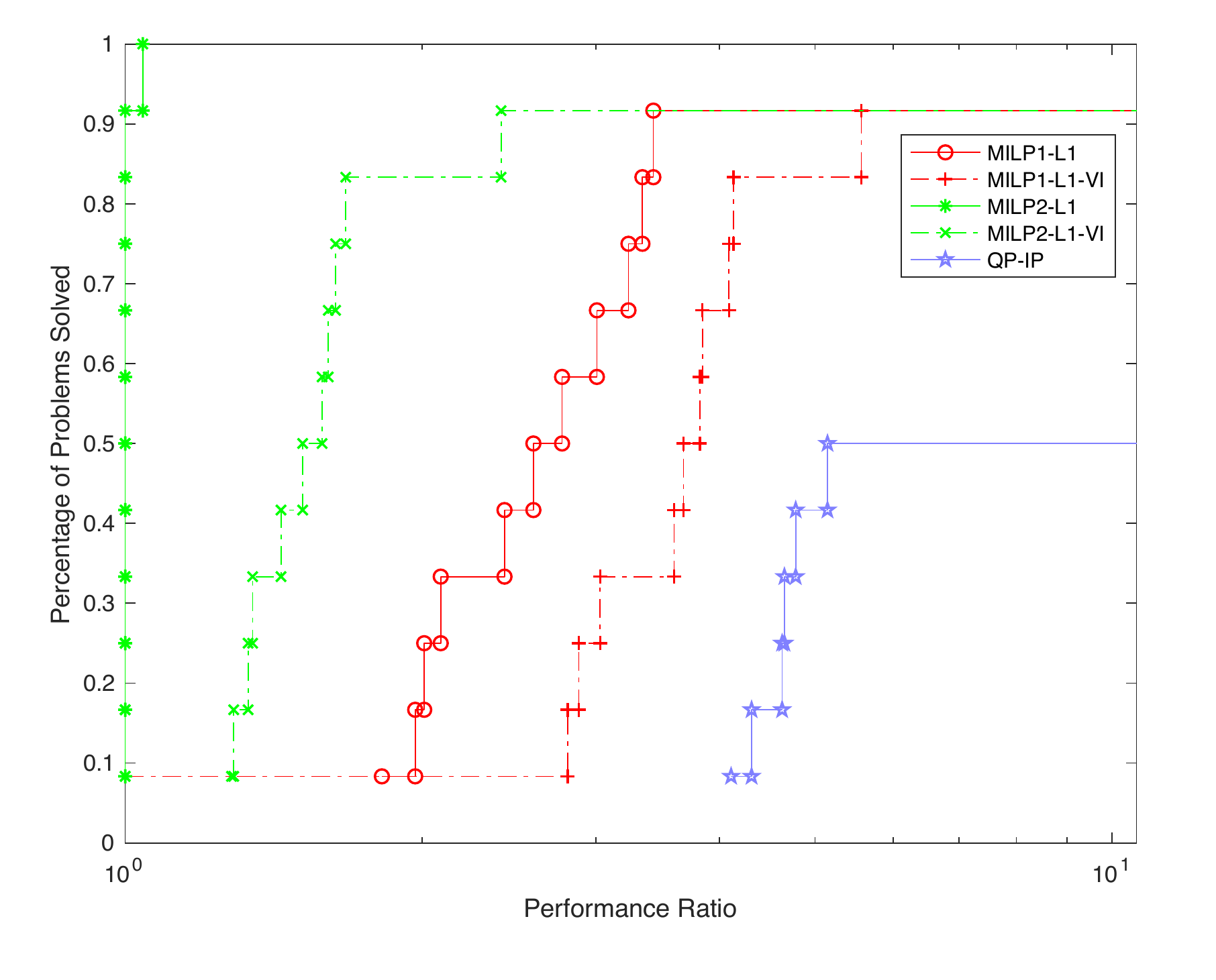}
\end{center}
\caption{Performance Profile on IS2}
\label{fig3}
\end{figure}

We present the performance profile for the larger instance set IS2 in Figure~\ref{fig3}. We do not report CPLEX QP and BARON since each of them hits the time limit on each instance in this set. Note again the logarithmic scale for the performance ratios.

Figure~\ref{fig3} illustrates that MILP2-L1 clearly outperforms each of the other MILP formulations on this set, which is then followed by MILP2-L1-VI, MILP1-L1, MILP1-L1-VI, and QP IP, respectively. 

These results illustrate that our MILP formulations constitute an effective approach for solving standard quadratic programs. The relaxed formulation (MILP2) seems to achieve the best solution times most frequently, especially when used with the improved lower bound $\ell_2(Q)$ whenever this bound can be computed. We remark, however, that the computational effort required for computing this tighter bound can be excessive on larger instances. The addition of valid inequalities may improve or degrade the performance depending on the instance and the number of such inequalities. However, for both (MILP1) and (MILP2), a larger number of instances can be solved to optimality within the time limit by the addition of valid inequalities. In summary, as illustrated by our computational results, both of our formulations are quite competitive, even when used with the simple lower bound $\ell_1(Q)$ and without valid inequalities. 

\section{Concluding Remarks} \label{conc}

In this paper, we propose solving standard quadratic programs by using two alternative MILP reformulations. The first MILP formulation arises from a simple manipulation of the KKT conditions. The second MILP formulation is obtained by exploiting the specific structure of a standard quadratic program and is in fact a relaxation of the first formulation. Both of our MILP formulations involve big-$M$ parameters. We derive bounds on these parameters by taking advantage of the particular structure of (StQP). We show that these bounds are functions of a lower bound on the optimal value. We consider two different lower bounds, which differ in terms of the computational effort and tightness. Our extensive computational experiments illustrate that the proposed MILP formulations outperform another recently proposed MILP approach in~\cite{XVZ15} and are much more effective than general purpose quadratic programming and nonlinear programming solvers. 

Since the tighter bound $\ell_2(Q)$ requires a considerable computational effort for large instances, one can use cheaper methods for computing or approximating this bound. For instance, the dual problem of (DNN) can be solved using a combination of proximal gradient method and binary search (see, e.g.,~\cite{KKT16}), which may be cheaper than using a semidefinite programming solver. Alternatively, based on the encouraging computational results reported in~\cite{LMP17}, a tight linear programming based lower bound can be employed in lieu of $\ell_2(Q)$. We leave these problems for future work.

Another interesting research direction is the investigation of decomposition approaches for the proposed MILP formulations arising from large-scale standard quadratic programs. In addition, encouraged by our computational results, we intend to identify other classes of nonconvex quadratic programs that would be amenable to a similar effective MILP formulation.

\section*{Acknowledgements}

Part of this research was performed while E. Alper Y{\i}ld{\i}r{\i}m was visiting the University of Edinburgh. The author gratefully acknowledges the hospitality of Prof.~Jacek Gondzio and the School of Mathematics. His visit was supported, in part, by T{\"U}B{\.I}TAK-B{\.I}DEB 2219 Turkish Scientific and Technological Research Council International Postdoctoral Research Scholarship Program, which is gratefully acknowledged.

\bibliography{dnn}{}
\bibliographystyle{plain}
\end{document}